\documentclass[12pt]{amsart}

\usepackage[totalwidth=15.75cm,totalheight=22.275cm]{geometry}
\usepackage{amsfonts,amsthm,amsmath,amssymb,amscd}
\usepackage{graphicx}
\usepackage{color}

\newcommand{\ol}[1]{\overline{#1}}
\newcommand{\T}{\mathcal{T}}
\newcommand{\V}{\mathcal{V}}
\newcommand{\Blog}{\mathcal{B}_{\log}}
\newcommand{\BlogP}{\Blog^{\operatorname{p}}}
\newcommand{\re}{\operatorname{Re}}
\newcommand{\Sequ}{\mathcal{S}}
\newcommand{\Sequt}{\tilde{\Sequ}}

\newcommand{\Ht}{\tilde{H}}

\newcommand{\taddr}{\underline{t}}
\newcommand{\saddr}{\underline{s}}

\newcommand{\extaddr}{\operatorname{addr}}
\newcommand{\addr}{\extaddr}
\renewcommand{\phi}{\varphi}
\newcommand{\eps}{\varepsilon}
\newcommand{\C}{\mathbb C}

\newcommand{\R}{\mathbb R}
\newcommand{\Z}{\mathbb Z}
\newcommand{\N}{\mathbb N}
\newcommand{\Q}{\mathbb Q}

\newcommand{\A}{\mathcal A}

\newcommand{\B}{\mathcal B}

\renewcommand{\Re}{\textup{Re}}

\newcommand{\Ch}{\hat{\C}}
\newcommand{\cantorbouquet}{Cantor bouquet}
\renewcommand{\P}{\mathcal{P}}
\newcommand{\ignore}[1]{}

\theoremstyle{plain}
\newtheorem{lem}{Lemma}[section]
\newtheorem{prop}[lem]{Proposition}
\newtheorem{thm}[lem]{Theorem}
\newtheorem*{thm*}{Theorem}
\newtheorem{cor}[lem]{Corollary}
\newtheorem*{cor*}{Corollary}

\theoremstyle{definition}
\newtheorem{defn}[lem]{Definition}
\newtheorem*{defn*}{Definition}

\newtheorem*{ex*}{Example}

\newtheorem*{rem*}{Remark}

\theoremstyle{remark}

 \newtheoremstyle{claimstyle}%
   {}
   {}
   {\normalfont}
   {}
   {\itshape}
   {.}
   { }
   {\thmnote{#3}}

\theoremstyle{claimstyle}
\newtheorem*{varclaim}{}

\newenvironment{remark}[1][Remark]{\begin{varclaim}[#1]}{\end{varclaim}}

\DeclareMathOperator{\dist}{dist}

\DeclareMathOperator{\Sing}{Sing}

\begin{document}

\title[Brushing the hairs of entire functions]{Brushing the hairs \\ of 
  transcendental entire functions}

\date{\today}

\author{Krzysztof Bara\'nski}
\address{Institute of Mathematics, University of Warsaw,
ul.~Banacha~2, 02-097 Warszawa, Poland}
\email{baranski@mimuw.edu.pl}

\author{Xavier Jarque}
\address{Departament d'Enginyeria Inform\`atica i Matem\`atiques,
Universitat Rovira i Virgili,
Avinguda Pa\"isos Catalans~26, 43007 Tarragona,
Catalunya, Spain}
\email{xavier.jarque@ub.edu}

\author{Lasse Rempe}
\address{Department of Mathematical Sciences,
University of Liverpool,
Liverpool, L69 7ZL, United Kingdom }
\email{L.Rempe@liverpool.ac.uk}

\thanks{The first author is supported by Polish MNiSW Grant N N201 0234 33 
and Polish MNiSW SPB-M. 
The second author is partially supported  by 
 MTM2006--05849/Consolider, MTM2008--01486 (including a FEDER
contribution), MTM2011-26995-C02-02, and by 2009SGR 792.
The third author
is supported by EPSRC fellowship EP/E052851/1. 
All three authors are supported by 
the EU FP6 Marie Curie Program RTN CODY}
\subjclass[2000]{Primary 37F10, Secondary 54F15, 54H20, 30D05}

\begin{abstract} Let $f$ be a
 transcendental entire
 function of finite order 
 in the Eremenko-Lyubich class $\B$
 (or a finite composition of such maps), 
 and suppose that $f$ is hyperbolic and has a unique Fatou 
 component. We show that the Julia set of $f$ is a
 \emph{\cantorbouquet}; i.e.\ is ambiently homeomorphic to a straight 
 brush in the sense of Aarts and Oversteegen. In particular,
 we show that any two such Julia sets are ambiently homeomorphic. 

 We also show that if $f\in\B$ has finite order (or is a finite composition
  of such maps), but is not necessarily hyperbolic with connected Fatou
  set,
  then the Julia set of $f$ \emph{contains} a \cantorbouquet. 

 As part of our proof, we describe, for an arbitrary function
  $f\in\B$, a natural compactification of the dynamical plane
  by adding a ``circle of addresses'' at infinity. 
\end{abstract}

\maketitle

\section{Introduction}\label{sec:intro}

In recent decades there has been an increasing interest in studying the 
 dynamics generated by the iterates of transcendental entire functions. 
 For these dynamical systems, the presence of an essential singularity
 at infinity creates some major differences to the dynamics of 
 polynomials or rational maps. 

 More precisely, let $f: \C \to \C$ be a transcendental
  entire  function.
 Then
 the \emph{Julia set} $J(f)$ is the set of points $z\in\C$ at which 
 the family  $\{f^n\}_{n > 0}$ is not equicontinuous with respect to the
 spherical metric 
 (this is where the dynamics of $f$ is ``chaotic''). The complement 
 $F(f)=\C\setminus J(f)$---i.e.\ the set of stable behaviour---
 is called the \emph{Fatou set}.
 Due to the essential singularity at infinity, the 
 Julia sets of transcendental entire functions
 tend to be far more complicated than for rational maps, and
 one of the main problems in the field is to understand these sets
 (and the dynamics thereon) from a topological and geometric point of view.

A cornerstone example of transcendental entire maps is given by the complex 
 exponential family $E_{\lambda}(z)=\lambda \exp (z),\ \lambda\in\C$.
 It is often considered the simplest possible family of transcendental
 entire functions, and has received considerable attention since at least
 the 1980s. 
 For real parameter values $\lambda \in (0, 1/e)$, the Fatou set has a 
 unique connected component, consisting of all points that 
 converge to an
 attracting fixed point $z_0^{\lambda}\in \R$ under iteration. 
 In \cite{DK} it was proved that for these parameters, the Julia set is given 
 by a union of pairwise disjoint arcs to $\infty$,
 which are called 
 {\it hairs} or {\it dynamic rays}. 
 Each point $z$ on each hair, except possibly the finite 
 endpoint of the hair, satisfies 
 $\Re\left(E_{\lambda}^j\right)(z) \to +\infty$ when $j \to \infty$.
 Due to its appearance, the union of these hairs has been referred to as
 a {\it \cantorbouquet} (a term to which, for the purposes of this article,
 we will give a precise mathematical meaning below). 
 For a long time, hairs as above have also been known to exist
 for many other functions, such as the sine family,
 $z\mapsto \lambda \sin(z)$; see \cite{DT}.

 Aarts and Oversteegen \cite{AO} gave a complete topological
  description of the Julia sets $J(E_{\lambda})$, $\lambda\in (0,1/e)$, as
  well as the Julia sets $J(\lambda\sin)$ for $\lambda\in(0,1)$, by
  showing that all these sets are homeomorphic to a single universal topological
  object. They did so by explicitly constructing 
  a homeomorphism between the Julia set and a type of subset of 
  $\R^2$ that is called a \emph{straight brush} (see Definition
  \ref{defn:brush}), and proved that
  any two straight brushes are homeomorphic. In fact, they even showed
  that the sets are \emph{ambiently} homeomorphic; i.e., the homemorphism
  between them can be chosen to extend to a homeomorphism of $\R^2$.
  (Compare also \cite{BDDJM,R1} for a discussion of the
  topological dynamics of other parameter values in the exponential family.) 
  As mentioned above, we shall refer to sets such as these
  as \emph{{\cantorbouquet}s}: 
\begin{defn}[\cantorbouquet]
 A \emph{\cantorbouquet} is any subset of the plane that is ambiently
  homeomorphic to a straight brush.
\end{defn}
\begin{remark}
 This is different from the terminology in \cite{DT}. There 
  a ``Cantor $n$-bouquet'' is a certain type of subset of the
  Julia set that is homeomorphic to the product of a Cantor set with
  the interval $[0,\infty)$, and a Cantor bouquet is simply the closure
  of an increasing sequence of Cantor $n$-bouquets, with $n\to\infty$. 

 We note that the Cantor bouquets we construct can also be seen to be 
  Cantor bouquets in the sense of Devaney and Tangerman, and that the
  functions considered in \cite{DT} satisfy the hypotheses of our 
  theorems. 
\end{remark}

 Recently, there has been significant progress in extending the above-mentioned
  description of exponential Julia sets (and those of other explicit maps) to
  much more general classes of transcendental entire functions. To state
  these results, we use the following standard definitions.
 \begin{defn}[Hyperbolicity]
  A transcendental
   entire function $f$ is called \emph{hyperbolic} if the
   \emph{postsingular set}
     \[ \P(f) := \ol{\bigcup_{j\geq 0} f^j(\Sing(f^{-1}))}  \]
   is a compact subset of the Fatou set $F(f)$. 

   The function $f$ is called \emph{of disjoint type} if it is
    hyperbolic and furthermore $F(f)$ is connected. (Equivalently,
    $\ol{\Sing(f^{-1})}$ is a compact subset of the immediate
    attracting basin of an attracting fixed point.)
 \end{defn}
 \begin{remark}
  Here $\Sing(f^{-1})$ denotes the set of \emph{singularities of $f^{-1}$}. 
   $\Sing(f^{-1})$ consists of the
   critical and asymptotic values of $f$. Recall that an
   \emph{asymptotic value} of $f$ is a point 
    $z_{0}\in\C$ for which there
    is a curve $\alpha(t):[0,\infty)\to\C$ satisfying  
   $|\alpha(t)| \to \infty$ and $f\left(\alpha(t)\right)\to z_{0}$ as 
   $t\to \infty$. For example, $0$ is the unique singular value of
   the exponential map $E_{\lambda}$. 

  Note that our definition of hyperbolicity implies, in particular,
   that the set $\Sing(f^{-1})$ is bounded. The set of all transcendental
   entire functions with the latter property is known as the
   \emph{Eremenko-Lyubich class} and denoted by $\B$. 
 \end{remark}

 \begin{defn}[Finite order]
  A transcendental entire function has \emph{finite order}
   if there are constants $c,\rho>0$ such that
   $|f(z)|\leq c\cdot e^{|z|^{\rho}}$ for all $z\in\C$. 
 \end{defn}

 The following result was established in \cite{B,RRRS}:

\begin{thm}[{\cite[Theorem C]{B}, \cite[Theorem 5.10]{RRRS}}]
  \label{thm:hairs}
 Let $f$ be a disjoint-type function of finite order. Then $J(f)$
  is a union of pairwise disjoint arcs to infinity.

 The conclusion holds, more generally, if $f$ is a disjoint-type function
  that can be written as a finite composition of finite-order functions
  in the class $\B$. 
\end{thm} 

 It was also claimed without proof
  in \cite{R,RRRS} that, under the hypotheses of the
  theorem, the Julia sets are {\cantorbouquet}s (i.e., ambiently
  homeomorphic to
  a straight brush in the sense of \cite{AO}). In this note, we justify
  this claim. 

\begin{thm} \label{thm:bouquet}
 Under the hypotheses of Theorem \ref{thm:hairs}, the Julia set
  $J(f)$ is a \cantorbouquet. 
\end{thm}

 We note that, as with the results in \cite{RRRS}, the 
  theorem holds for a more general class of functions
  than stated above (see Definition \ref{defn:headstart} and 
  Corollary \ref{cor:bouquetheadstart}).
  However, it is necessary to make some function-theoretic assumptions
  in addition from the dynamical requirement of the function being of
  disjoint type. 
  Indeed,
  in \cite[Theorem 8.4]{RRRS}, 
  it is shown that there exists a disjoint-type function
  $f$ such that $J(f)$ contains no nontrivial
  curves. So in this case $J(f)$ is very far from being
  a \cantorbouquet! 

 Theorem \ref{thm:hairs} and Corollary \ref{cor:bouquetheadstart}
  imply, in particular, that all the Julia sets covered by these
  results are ambiently homeomorphic. 
  We note that this is in contrast to the
  fact---already noted in \cite{AO}---that the geometry and dynamics
  of Cantor bouquets can be rather diverse. We mention a few examples:
  \begin{itemize}
   \item For hyperbolic exponential maps, the Julia set has 
    zero Lebesgue measure, but for 
    maps in the sine family, the area is positive \cite{mcmullen}. 
   \item If $f$ is of finite order, then the Julia set always has
    Hausdorff dimension $2$, while the dimension of the set of hairs
    (without endpoints) is equal to $1$ \cite{B2}. On the other hand,
    Stallard \cite{stallard} showed that 
    the Hausdorff dimension of $J(f)$ can take any value
    $d\in (1,2]$,  and it can be checked that these examples
     satisfy the hypotheses of
     Corollary \ref{cor:bouquetheadstart}. 
  \item For hyperbolic exponential maps, 
   the set of escaping points has Hausdorff dimension $2$,
   and the set of non-escaping points has Hausdorff dimension strictly less than
   $2$ \cite{UZ}. However, there are examples of finite-order disjoint-type 
   functions for which the set of nonescaping points also has full Hausdorff dimension
   \cite{hypdim_two} and functions (of infinite order) for which the
   escaping set has Hausdorff dimension equal to $1$, and the Julia set has
   dimension arbitrarily close to $1$ \cite{rempestallard}. 
 \end{itemize}

\medskip
 It has been known for a long time that hairs exist in the dynamical
  plane of exponential maps also when the map is not hyperbolic. Indeed,
  in recent years this approach has been instrumental in obtaining
  a 
  much deeper understanding of the dynamics for arbitrary
  parameter values of the exponential family 
  (compare \cite{R1,SZ,SZ1} and references therein). 

 The results in \cite{RRRS} actually establish the existence of hairs
  for all (not necessarily disjoint type or hyperbolic) maps
  $f\in\B$ of finite order, as well as their compositions. Essentially, 
  it is proved that the set of all points that stay near infinity under
  the iteration of such a function consists of hairs. (See 
  Corollary~\ref{cor:absorbing}
  below for a precise statement.) We prove a version of our theorem
  also in this case, which shows in particular that the Julia 
  set of such a function always contains a \cantorbouquet.

 \begin{thm}[Absorbing {\cantorbouquet}s]\label{thm:bouquetgeneral}
  Let $f$ be a finite order function in the class $\B$ 
   (or, more generally, a composition of such
   functions). Then for every $R>0$, there exists 
   a {\cantorbouquet} $X\subset J(f)$ with $f(X)\subset X$ such that 
  \begin{enumerate}
   \item $|f^j(z)|\geq R$ for all $z\in X$ and $j\geq 0$;
   \item there is $R'\geq R$ such that, if $z\in\C$ with
      $|f^j(z)|\geq R'$ for all $j$, then $z\in X$. 
  \end{enumerate}
 \end{thm}

\medskip

In our proof we use the results from \cite{B,RRRS}, together with
 the topological characterisation of {\cantorbouquet}s given in \cite{AO}.
 To apply the latter, we compactify the Julia set $J(f)$ by adding a 
 circle (of ``addresses'') at infinity, with different hairs of $J(f)$
 ending at different points of this circle. This construction is, in fact,
 completely general: If $f\in\B$ is arbitrary, then we describe
 a natural combinatorial compactification of the dynamical plane, which
 is likely to be useful in future applications. (Again, this generalizes
 a construction well-known in the setting of exponential maps.)

\subsection*{Organization of the paper}
 In Section~\ref{sec:hairy} we recall the topological definitions and results
  from \cite{AO}. In Section \ref{sec:log}, we shortly review the 
  \emph{logarithmic change of coordinates}, one of the main tools for
  studying the class $\B$, and the combinatorial notion of external addresses.
  Section~\ref{sec:hairs} contains a description of the hairs from 
  Theorem \ref{thm:hairs}, while Section~\ref{sec:comb} constructs a
  suitable 
  compactification of the Julia set. Finally,
  we prove our main theorems in
  Sections \ref{sec:proof} and \ref{sec:proofgeneral}.

 \subsection*{Basic notation}
  As usual, we denote the complex plane by $\C$ and the Riemann sphere by
   $\Ch=\C\cup\{\infty\}$. Closures and boundaries will be considered in
   $\C$ unless explicitly stated otherwise.

  We use the notation $\N_0$ for the set of
  nonnegative integers.

\section{Straight brushes and hairy arcs}\label{sec:hairy}
 In this section, we review some topological notions and results from
  \cite{AO}. We begin with the formal definition of a straight brush. 

\begin{defn}[Straight brush] \label{defn:brush}
A subset $B$ of $[0,+\infty) \times (\R\setminus\Q)$ is called a 
  \emph{straight brush}
  if the following properties are satisfied:
\begin{itemize}
\item
The set $B$ is a closed subset of $\R^2$.
\item 
For every $(x,y) \in B$ there exists $t_y \geq 0$ such that $\{x: (x,y)\in B\} = [t_y, +\infty)$. 
The set $[t_y, +\infty) \times \{y\}$ is called the hair attached at $y$ and the point $(t_y, y)$ is called its endpoint.
\item
The set $\{y: (x,y) \in B \text{ for some } x\}$ is dense in $\R\setminus\Q$. 
  Moreover, for every $(x, y) \in B$ there exist two sequences of hairs 
  attached respectively at $\beta_n, \gamma_n \in  \R\setminus\Q$ such that 
  $\beta_n < y < \gamma_n$, $\beta_n, \gamma_n \to y$ and 
  $t_{\beta_n}, t_{\gamma_n} \to t_y$ as $n\to\infty$.
\end{itemize}
\end{defn}

A simple topological characterization of straight brushes can be 
 found in \cite{BO}; in particular, any two such spaces are homeomorphic.
 In fact, as shown in \cite{AO},
 they are even \emph{ambiently} homeomorphic, which essentially
 requires that the ``vertical'' order of the hairs is preserved under
 the homeomorphism. To discuss this, it is useful to 
 consider a compactification of the brush instead. Hence we make
 the following definitions, following \cite{AO}.

\begin{defn}[Comb]
 A comb is a continuum (i.e., a compact connected metric space) $X$
  containing an arc $B$ (called the base of $X$) such that:
 \begin{enumerate}
  \item The closure of every connected component of $X\setminus B$ is an arc,
         with exactly one endpoint in $B$.
         In particular, for every $x\in X\setminus B$, there exists
          a unique arc $\gamma_x:[0,1]\to X$ with $\gamma_x(0)=x$,
          $\gamma_x(t)\notin B$ for $t<1$ and $\gamma_x(1)\in B$. 
          In this case, we say that $x$ belongs to the hair
          attached at $\gamma_x(1)$. 
  \item Distinct components of $X\setminus B$ have disjoint closures. 
  \item The set $X\setminus B$ is dense in $X$. 
  \item If $x_0\in X\setminus B$ and $x_n\in X\setminus B$ is a sequence
         of points converging to $x_0$, then
         $\gamma_{x_n}\to \gamma_{x_0}$ in the Hausdorff metric as
         $n\to\infty$. 
 \end{enumerate}
\end{defn}
\begin{remark}
 This definition is formally somewhat different from the one
  given in \cite{AO}, but it is not difficult to see that they are
  equivalent. We decided to use this, rather explicit,
  form to avoid having to
  introduce additional topological concepts that are not relevant
  to the remainder of the paper.
\end{remark}

\begin{defn}
 A \emph{hairy arc} is a comb with base $B$ and a total order 
  $\prec$ on $B$ (generating the 
   topology of $B$) such that the following holds. 
  If $b\in B$ and $x$ belongs to the hair attached at $b$, then
  there exist sequences $x_n^+,x_n^-$, attached respectively at
  points $b_n^+,b_n^-\in B$, such that 
  $b_n^-\prec b \prec b_n^+$ and $x_n^-,x_n^+\to x$ as 
  $n\to\infty$.

 A set $A\subset \R^2$ is called a \emph{one-sided hairy arc} if
  $A$ is topologically a hairy arc and all hairs are
  attached to the same side of the base. 
\end{defn}
\begin{remark}[Remark 1]
 Here it should be intuitively clear what ``being attached to the same side of the base''
  means (this is the same terminology as used in \cite{AO}). Formally, we can define
  this notion as follows: there exists a Jordan arc $C\in \R^2\setminus A$ connecting 
  the two endpoints of the base $B$ such that all connected components of
  $A\setminus B$ belong to the same connected component of $\R^2\setminus (B\cup C)$.
\end{remark}
\begin{remark}[Remark 2]
 In \cite{AO}, a hairy arc is defined, instead, as a space
  homeomorphic to what is called a \emph{straight one-sided hairy arc}
  (in analogy with the notion of a straight brush). 
  It is then shown in 
  \cite[Theorem 3.11]{AO} that this coincides with the
  topological definition above. 
\end{remark}

 It is easy to see that, to any straight brush, we can add
  a base $B=\{(\infty,y): y\in [-\infty,+\infty]\}$ to obtain
  a hairy arc. Conversely, for any hairy arc $X$, the
  set $X\setminus B$ is homeomorphic to a straight brush; this
  is implied by the following theorem. 

\begin{thm}[{\cite[Theorems 3.2, 3.11 and 4.1]{AO}}] \label{thm:ambient}
 Any two hairy arcs are homeomorphic. 

 Furthermore, any two
  one-sided hairy arcs $X_1,X_2\subset \R^2$ are ambiently homeomorphic.
  More precisely, any homeomorphism $\phi:X_1\to X_2$ extends
  to a homeomorphism $\phi:\R^2\to \R^2$.

 In particular, any two
   straight brushes are ambiently homeomorphic.
\end{thm}

\section{Logarithmic coordinates and the class $\BlogP$} \label{sec:log}

A useful tool for studying functions in the Eremenko-Lyubich class $\B$ is the
 \emph{logarithmic change of coordinates}. 
 Let $f\in\B$ and let $D\subset\C$ be a Jordan domain
 that contains the set $\overline{\Sing(f^{-1})}$ as well as
 the values $0$ and $f(0)$. 
 Set $W := \C\setminus\overline{D}$. The connected components of 
  $\V := f^{-1}(W)$ are called the tracts of $f$. Each tract $T$
  is a Jordan domain whose boundary passes through $\infty$, and $f:T\to W$ 
  is a universal covering. If $f$ has finite order, then the
  number of tracts will be finite, but otherwise there may be
  infinitely many tracts (as is the case e.g.\ for $z\mapsto e^{e^z}$).
 
It is convenient to study $f$ in logarithmic coordinates. To do so, we set
  $H := \exp^{-1}(W)$ and $\T := \exp^{-1}(\V)$. 
  Each component of $\T$ is a simply connected domain whose boundary
  is  homeomorphic to $\R$ 
  such that both 'ends' of the boundaries have real parts converging to 
  $+\infty$. 
  Note that $\T$ and $H$ are invariant under translation by $2\pi i$.
 We can lift $f$ to a map 
 \[
  F: \T \to H,
 \]
 satisfying $\exp\circ F = f\circ\exp$. We may also assume that the function
 $F$ is chosen to be $2\pi i$-periodic. 
 We call  
 $F$ a \emph{logarithmic transform} of $f$.  By construction, 
 the following properties hold.

   \begin{enumerate}
    \item $H$ is a $2\pi i$-periodic unbounded Jordan domain that contains
     a right half-plane.   \label{item:H}
    \item $\T\neq\emptyset$ 
     is $2\pi i$-periodic and $\re z$ is bounded from below in
     $\T$.
    \item Each component $T$ of $\T$ is an 
      unbounded Jordan domain that is disjoint
      from all its $2\pi i\Z$-translates. 
      For each such $T$, the restriction
      $F:T\to H$ is a conformal 
      isomorphism whose continuous extension to the closure
      of $T$ in $\Ch$ satisfies $F(\infty)=\infty$. 
      ($T$ is called a \emph{tract of $F$}; we 
       denote the inverse of $F|_T$ by 
       $F_T^{-1}$.) \label{item:tracts}
    \item The components of $\T$ have pairwise disjoint closures and
      accumulate only at $\infty$; i.e.,
      if $z_n\in\T$ is a sequence of points all belonging to different
      components of $\T$, then $z_n\to\infty$.  \label{item:accumulatingatinfty}
    \item $F$ is $2\pi i$-periodic. \label{item:Fperiodic}
   \end{enumerate}

   Following \cite{RRRS}, let us denote 
    by $\BlogP$ the class of all functions
    \[ F:\T\to H, \]
   where $F$, $\T$ and $H$ have the properties
   (\ref{item:H}) to (\ref{item:Fperiodic}), regardless of whether
    $F$ arises as the logarithmic transform of a function $f\in\B$ or not. 
   The advantage of formulating our results in this more general setting
   is that they can also be applied e.g.\ to meromorphic functions having
   logarithmic singularities over infinity, while the proofs remain unchanged.  
  \begin{remark}
   To avoid confusion, we should note that, in \cite{R}, the notation $\Blog$
    is used for the class we denote by $\BlogP$, while in
    \cite{RRRS} and other subsequent papers, $\Blog$ denotes
    the class of functions satisfying
    (\ref{item:H}) to (\ref{item:accumulatingatinfty}), 
    but not necessarily
    (\ref{item:Fperiodic}). 
    While many of the results\footnote{
      The exceptions are
      Proposition \ref{prop:density} and those facts that depend on it,
      namely Corollary \ref{cor:bouquetheadstart} and the results of
      Section \ref{sec:proofgeneral}.} 
    we state for the class $\BlogP$ also hold in this
    more general setting, our main theorems 
    do not. (The condition of periodicity in (\ref{item:Fperiodic})
    could, however,
    be considerably weakened, at the expense of a more technical definition.)
  \end{remark}

  For $F\in\BlogP$, we define 
    \[ J(F) := \{z\in\ol{\T}: F^j(z)\in \ol{\T}\text{ for all $j\geq 0$} \} \] 
  Note that the continuous extension of $F|_T$ in item (3) exists 
  by Carath\'{e}odoryÕs Theorem (see \cite[Theorem 2.1]{CG}). Because $T$ is a Jordan domain, this extension is a homeomorphism, 
and in particular $F|_T$ extends continuously to a homeomorphism between the closures $\overline{T}$ and 
$\overline{H}$ (taken in $\mathbb{C}$). Together with property (4), this implies that $F$ extends continuously to the 
closure $\overline{\mathcal{T}}$ of $\mathcal{T}$ in $\mathbb{C}$.
   In analogy with our notation for entire functions in the class $\B$, we say
    that a function $F\in\BlogP$ is of \emph{disjoint-type} if 
    the boundaries of the tracts of $F$ do not intersect the boundary of $H$;
    i.e.\ if $\ol{\T}\subset H$. The following properties of
    disjoint type functions are elementary and well-known. 
  
  \begin{lem}[Properties of disjoint-type functions] 
   If $F:\T\to H$ is of disjoint type, then $F$ is uniformly expanding
    with respect to the hyperbolic metric in $H$. That is, there is
    a constant $C>1$ such that the derivative $\|DF(z)\|$ of
    $F$ with respect to the hyperbolic metric satisfies
    $\|DF(z)\|\geq C$ for all $z\in \ol{\T}$. 

   If $f\in\B$, then the following are equivalent:
    \begin{enumerate}
     \item $f$ is of disjoint type;
     \item $\Sing(f^{-1})$ is compactly contained in the immediate basin  
      of attraction of an attracting fixed point of $f$;
     \item $f$ has a logarithmic
       transform $F$ of disjoint type (in which case  we have
       $J(f) = \exp(J(F))$);
    \end{enumerate}
  \end{lem}
  \begin{proof}
    The first claim is a consequence of the expansion estimate for
     functions in the class $\BlogP$ (see Lemma \ref{lem:eremenkolyubich} below); compare
     \cite[Lemma 3]{B} or \cite[Lemma 2.1]{RRRS} for details.

    For a proof of the second part, see 
     \cite[Proposition 2.7]{helenaconjugacy}.
  \end{proof}

 \subsection*{External addresses}
  The (Markov-type) partition of the domain of a function $F\in\BlogP$ by 
   the tracts of $F$ allows us to define symbolic dynamics. 

  \begin{defn}
   Given a function $F\in\BlogP$, we denote by $\A$ the set of tracts of $F$.
    This set $\A$ is called the \emph{symbolic alphabet} associated to $F$. 
  \end{defn}
  We note that, for practical purposes and also in other
   applications, it is
   usually more convenient and appropriate to use a symbolic set as
   address entries, rather than the set of tracts of $F$ itself. For example,
   suppose that $F$ is the logarithmic transform of a function
   $f\in\B$ which has finitely many, say $N$, tracts. Then
   $\A$ can be identified with the set of pairs $(s,j)$, where
   $s\in\Z$ and $j\in \{1,\dots,N\}$. Here $j$ identifies one of the tracts
   of $f$, and $s$ determines one of the countably many logarithmic lifts
   of this domain. 
   However, for our purposes it seems more natural and elegant to use the
   definition given above. 

 \begin{defn}[External addresses]
  Let $F\in\BlogP$, and let $z\in J(F)$. For each $j\geq 0$, let
   $T_j\in\A$ be the (unique) tract of $F$ with $F^j(z)\in \ol{T_j}$. 
   Then the sequence 
   $\extaddr(z) := T_0 T_1 T_2\dots$ is called the
   \emph{external address of $z$}. 

  Furthermore, we will refer to 
   any sequence $\saddr= T_0 T_1 T_2\dots \in \A^{\N_0}$ as an 
   \emph{(infinite) external
   address (for $F$)}, regardless of whether it is realized as the address 
   of a point $z\in J(F)$ or not. 
   For such an address $\saddr$, we denote the set of points
   $z\in J(F)$ with $\extaddr(z)=\saddr$ by 
    \[ J_{\saddr}(F) := \{z\in J(F): F^j(z)\in \ol{T_j}\text{ for all $j\geq 0$}\}.
              \]
 \end{defn}
 \begin{remark}
  External addresses are also referred to as \emph{itineraries}.
   We prefer not to use this terminology, on the one hand to
   stress the analogy with \emph{external angles} from polynomial dynamics,
   and on the other in order to not cause confusion with
   other types of itineraries (defined with respect to certain dynamical
   partitions). 
 \end{remark}

 The alphabet $\A$ is equipped with a natural order, corresponding to
  the \emph{vertical order} 
  (see \cite[p.325]{expcombinatorics})
  of the tracts of $F$ near infinity. 
  This gives rise to the \emph{lexicographic order} on the set
  $\A^{\N_0}$ of external addresses, which we also denote by ``$<$''. 

 If $f\in\BlogP$ and $F$ is a logarithmic transform of $f$, then
  we can also define \emph{external addresses for $f$} by identifying
  any two addresses $\saddr= T_0 T_1 T_2 \dots$ and $\saddr' = T_0' T_1' T_2'$ of
  $F$ with $T_0'=T_0+2\pi i k$ for some $k\in \Z$ and $T_j=T_j'$ for
  $j>0$. Note that this space is no longer linearly ordered, but carries
  a circular ordering. For a more natural (but equivalent) way of
  defining external addresses of $f$, without using logarithmic
  coordinates, see e.g.\ \cite[Section 2.3]{helenalanding}. 

 The following fact implies that there is a dense
  set of addresses for which 
  $J_{\saddr}(F)$ is nonempty.
 \begin{lem}[{\cite[Proposition 2.4]{R}}] \label{lem:periodic}
  Suppose that $\saddr\in\A^{\N_0}$ is a periodic address (i.e.,
   $s_{n+k}=s_n$ for some $k>0$ and all $n\geq 0$). Then
   $J_{\saddr}(F)\neq\emptyset$.
 \end{lem}
 \begin{remark}
  Proposition 2.4 in \cite{R} is stated for fixed addresses; the
   above statement follows by considering the iterate $F^k$. It is
   not difficult to see that the result is also true when the 
   address $\saddr$ contains only finitely many symbols. 
 \end{remark}

 \subsection*{Accumulation}
  Recall that a key property in the definition
   of a hairy arc is that
   no hair is isolated either from above or from below. For
   disjoint-type functions, 
   this will be immediate from the following simple fact.

 \begin{prop}[Accumulation from above and below]   \label{prop:density}
  Let $F\in\BlogP$ be of disjoint type, and let $z_0\in J(F)$.
   Then there exist sequences $z^-_n,z^+_n\in J(F)$ with
   $\addr(z^-_n)<\addr(z_0)<\addr(z^+_n)$ for all $n$ and
   $z^-_n\to z_0$, $z^+_n\to z_0$. 
 \end{prop}
 \begin{proof}
  Let $n\geq 1$, and let $\phi:H\to H$ be the branch of $F^{-n}$ that maps
   $F^n(z_0)$ to $z_0$. We define 
    \[ z^{\pm}_n := \phi(F^n(z_0)\pm 2\pi i). \]
  Then, by definition, we have $\addr(z^-_n)<\addr(z_0)<\addr(z^+_n)$ for
   all $n$. 
  Since $H$ contains a right half plane, and because $H$ and $J(F)$ are
   $2\pi i$-periodic, the hyperbolic distance 
   $\dist_H(z,z+2\pi i)$ is uniformly bounded independently of
   $z\in J(F)$. Since $F$ is a strict hyperbolic expansion, it follows that
   $z^{\pm}_n\to z_0$ as $n\to \infty$, as claimed. 
 \end{proof}

\section{Existence of hairs: Head-start conditions}  \label{sec:hairs}
 Since our proof of the main theorem will strongly use the characterization
  of hairs from the proof of 
  Theorem \ref{thm:hairs}, we shall review some of the definitions
  and concepts 
  here. This will also allow us to state our results in a more
  general form, using the notions from \cite{RRRS}.

  The essential idea of the proof is the following. Suppose
   $F$ is the logarithmic transform of a function of finite order.
   Then for any address $\saddr$,
   the set $J_{\saddr}(F)$ is naturally ordered according to
   ``escape speed'': if $z,w\in J_{\saddr}(F)$ are such that
   $\re w$ is significantly larger than $\re z$, then
   the same will be true for $F(w)$ and $F(z)$
   ($w$ has a ``head start'' over $z$). This idea is formalized
   in the following definition.

  \begin{defn}[Head-Start Condition]\label{defn:headstart}
 Let $F\in\BlogP$, and let 
    $\phi:\R\to\R$ be a (not necessarily strictly) increasing 
  continuous function with
  $\phi(x)>x$ for all $x\in\R$. We say that
   $F$ satisfies the 
  {\em uniform head-start condition for $\phi$} if:
 \begin{enumerate}
  \item For all tracts $T$ and $T'$ of $F$ and all 
     $z,w\in \ol{T}$ with $F(z),F(w)\in \ol{T'}$, 
  \[
    \re w >\phi(\re z)\ \Longrightarrow\ \re F(w)>\phi(\re F(z))\;.
   \]
  \item For all external addresses $\saddr\in\A^{\N_0}$ and 
    for all distinct
    $z,w\in J_{\saddr}$,  there 
    is $M\in\N_0$ such that 
    $\re F^{M}(z)> \phi(\re F^{M}(w))$ or
    $\re F^{M}(w) > \phi(\re F^{M}(z))$.
 \end{enumerate}
\end{defn}

The above-mentioned fact for finite-order functions can now be phrased
 as follows.

\begin{prop}[{\cite[Lemmas 1 and 4]{B} or \cite[Section 5]{RRRS}}]
 Let $f$ belong to the class $\B$ and have finite order, 
  or more generally be a finite composition
  of such functions. Then there are a logarithmic transform
  $F:\T\to H$ for $f$ and constants $M>1$, $K>0$
  such that $F$ satisfies the uniform head-start condition
  for $\phi(x) = M\cdot x + K$.

 If $f$ is of disjoint type, then $F$ can also be chosen of disjoint
  type. 
\end{prop}

 As mentioned above, the proof of Theorem \ref{thm:hairs} relies on the idea
  of a ``speed ordering'' generated by the head-start condition. 
  More precisely:

 \begin{prop}[{\cite[Propositions 4.4 and 4.6]{RRRS}}]\label{prop:speedordering}
  Let $F\in\BlogP$ satisfy a uniform head-start condition, say for $\phi$, 
   and let $\saddr$ be an external address. We define
   a relation $\prec$ on $J_{\saddr}(F)\cup\{\infty\}$---the ``speed ordering''---%
   by
  \[ z\prec w \quad \Longleftrightarrow \quad
       \exists j\geq 0: F^j(w)> \phi(F^j(z)) \]
   and $z\prec\infty$ for all $z\in J_{\saddr}(F)$. 

  Then $\prec$ is a total order, and the order topology coincides with
   the usual topology on $J_{\saddr}(F)\cup \{\infty\}$. In particular,
   every connected component of $J_{\saddr}(F)\cup\{\infty\}$ is an arc. 

  Moreover, if $J_{\saddr}(F)\neq\emptyset$, then
   $J_{\saddr}(F)$ has a unique unbounded component $J_{\saddr}^{\infty}(F)$,
   and for every $z\in J_{\saddr}^{\infty}(F)$, the set
   $\{w\in J_{\saddr}(F)\cup\{\infty\}: z\prec w\}$ is an arc connecting $z$ to
   infinity. 
 \end{prop} 

 Note that
  the above result immediately
  implies Theorem \ref{thm:hairs}, since $J_{\saddr}(F)$ is connected when
  $F$ is of disjoint type. To cover also the case where
  $F$ is not of disjoint type, let us state the following
  result, which is a consequence of the preceding proposition. 

\begin{cor}[Hairy Absorbing Sets {\cite[Theorem 4.7]{RRRS}}]
\label{cor:absorbing}
 Suppose that $F\in\BlogP$ satisfies a uniform head-start condition. Then 
  there exists a closed $2\pi i$-periodic
   subset $X\subset J(F)$ with the following
  properties:
  \begin{enumerate} 
   \item $F(X)\subset X$;
   \item each connected component $C$ of $X$ is a closed arc to infinity;
     \label{item:arcs}
   \item there exists $K'>0$ with the following property:
      If 
      $z\in J(F)$ such that $\re F^j(z)\geq K'$ for all $j\geq 0$, then
      $z\in X$. 
             \label{item:absorbing}
  \end{enumerate}
  If 
   $F$ is of disjoint type, then we may choose $X=J(F)$. 
\end{cor}

 The goal of this paper is to show that $X$ can be chosen to be
  a {\cantorbouquet}. 
 
\section{Dynamical compactification} \label{sec:comb}

 We now wish to compactify the Julia set of a disjoint-type entire
  function as in our theorem in such a way that the resulting
  set will turn out to be a one-sided hairy arc (where the
  original Julia set coincides precisely with the union of
  hairs). 

 Essentially, this is done as follows: It is easy to compactify
  the space of external addresses, by adding 
  ``intermediate addresses'', such that the resulting space is
  homeomorphic to an arc. The desired space is then obtained by
  using this arc of addresses as the base of a comb, with the
  hair at address $\saddr$ attached at $\saddr$. 

 In fact, such a compactification of the dynamical plane can be defined
  completely generally for functions in the class $\B$ (or $\BlogP$),
  without assuming that the escaping set is a union of hairs, or that
  the function is hyperbolic. This type of compactification
  has already been used extensively in the exponential family;
  compare e.g.\ \cite[Section 2]{nonlanding} and
  \cite[Section 2]{expcombinatorics}, although it does not seem clear whether
  the precise details have been published previously.
  It is likely that this construction will be useful in future
  applications, and we hence use this opportunity to give a detailed
  account in the general setting.

 The construction 
  consists of two steps: Firstly, we must compactify the
  space of addresses $\A^{\N_0}$ 
  for a given map $F\in\BlogP$ to the desired
  arc $\Sequt$. Then we define a topology on 
  $\Ht := \ol{H}\cup \Sequt$ with the desired properties. 

  Finally, we shortly discuss the corresponding results in the
   original dynamical plane of an entire function $f\in\B$
   (for future reference). 
  
 \smallskip

  We begin with the construction of the compactification
   $\Sequt$.

 \begin{thm}
  There exists a totally ordered set $\Sequt\supset \A^{\N_0}$
   (where the order on $\Sequt$ agrees with lexicographic order
   on $\A^{\N_0}$) with the following properties:
  \begin{enumerate}
   \item With the order topology, $\Sequt$ is homeomorphic to a line segment.
     In particular, $\Sequt$ is compact and order-complete 
     (i.e., every nonempty subset that is bounded from above
      has a supremum). 
   \item $\A^{\N_0}$ is dense in $\Sequt$. 
  \end{enumerate}

  (Furthermore, this compactification is unique; i.e.\ 
    any other  ordered set with these properties is order-isomorphic to
    $\Sequt$, with the order-isomorphism restricting to the identity
    on $\A^{\N_0}$.) 
 \end{thm}
 This result is actually (a special case of)
  a standard fact from the theory of ordered spaces. However, it is possible---%
  and useful in the following---to give an explicit description of the
  space $\Sequt$. We comment below 
  on the particularly simple case where the number of 
  tracts
  is finite up to translations by multiples of $2\pi i$, as is the case for
  functions of finite order.

 To begin, let $F:\T\to H$ be a function in the class $\BlogP$. Recall
  that the symbolic alphabet 
  $\A$ is the set of tracts of $F$ (that is, the set of
  components $T$ of $\T$), equipped
  with a natural total order. Note also that $\A$ is countable. 

 Now we form an extended (also totally ordered) 
  alphabet $\ol{\A}$ by adding ``intermediate'' entries, corresponding to
  ``Dedekind cuts'' of the set $\A$.

 \begin{defn}[Intermediate entries]
  An \emph{intermediate entry} is a pair $(A^-,A^+)$ with
   $\A=A^-\cup A^+$,
    $A^-, A^+\neq\emptyset$
    and $a^-<a^+$ for all $a^-\in A^-$ and $a^+\in A^+$. 

  We define $\ol{\A}$ to be the union of $\A$ and the set of all intermediate
   entries. 
 \end{defn}

 \begin{lem}[Total order on {$\ol{\A}$}]
   There is a natural total order on $\ol{\A}$ that agrees with the
   vertical order on $\A$, and $\ol{\A}$ is order-complete.
 \end{lem}

 Note that, when there are only finitely many tracts (up to $2\pi i\Z$-translations),
  this construction simply adds an intermediate entry between any pair of adjacent
  tracts. Otherwise, there will also be intermediate entries corresponding to ``limit points''
  of sequences of tracts.

 It may be worth pointing out that the construction differs from the standard construction of
  the real numbers from the rationals using Dedekind cuts, in that every element
  $a\in\A$ will be isolated in the space $\ol{\A}$. More precisely, there are two intermediate
  entries, above and below $a$, which separate $a$ from all other elements of $\A$. 

 Nonetheless, the proof of order-completeness remains essentially the same:
  Let $B\subset\ol{\A}$ be bounded and nonempty; we must show the existence
  of $\sup B$. If $B$ has a maximal element, then we are done. Otherwise,
  set 
   \[ A^+ := \{a\in \A: b<a\text{ for all $b\in B$}\}\quad\text{and}\quad
      A^- := \A\setminus A^+. \]
  Then it is easy to check that
   $(A^-,A^+)=\sup B$. 

 \smallskip

 We can now describe the compactification $\Sequt$ of the space of addresses as follows.

 \begin{defn}
 A (finite) sequence
   $\saddr = T_0 T_1 \dots T_{n-1} S_n$, where $n\geq 0$,
   $T_j\in \A$ for $0\leq j\leq n-1$, and where $S_n\in \ol{\A}\setminus\A$ is an intermediate entry,
  is called an \emph{intermediate external address}. 

  We define $\Sequ$ to be the union of the set
   $\A^{\N_0}$ of all external addresses and the set of all intermediate external addresses. 
 \end{defn}

\begin{lem}
Then (with 
   respect to lexicographic order), the space $\Sequ$ is order-isomorphic to $\R$. 

  In particular, the space $\Sequt := \Sequ\cup\{-\infty,+\infty\}$ is compact when equipped with
   the order topology (where $-\infty<\saddr<+\infty$ for all $\saddr\in\Sequ$), and homeomorphic to   
   $[0,1]$. 
 \end{lem}
 \begin{proof}
  It is not difficult to see that the space $\Sequt$ is compact, metrizable and connected, and hence 
   homeomorphic and order-isomorphic to $[0,1]$ by a general theorem of topology \cite[Theorems 6.16 and 6.17]{nadler}. Instead, 
   we sketch a simple direct proof with a dynamical flavor. 

  Because $\A$ is countable, and by definition of $\ol{\A}$, we can partition the real line $\R$ into sets
   $I(a)\subset\R$, $a\in\ol{\A}$ such that:
  \begin{itemize}
   \item For each $a\in\A$, the set $I(a)$ is an open interval.
   \item For each $a\in\ol{\A}\setminus\A$, the set $I(a)$ is a singleton.
   \item The sets $I(a)$ are pairwise disjoint and their order on the real line coincides with the
     order on $\ol{\A}$. That is, if $a,b\in\ol{\A}$ and $a<b$, then $x<y$ for all $x\in I(a)$ and $y\in I(b)$. 
   \item $\bigcup_{a\in\ol{\A}} I(a)=\R$. 
  \end{itemize}
  (This can be done by identifying $\A$ with a countable discrete---%
    but not necessarily closed---subset of $\R$, and choosing an
    interval $I(a)$ around each $a\in\A$ such that the complement of the union of these intervals
    consists of singletons. We leave the details to the reader.)

   Now define a continuous function 
     \[ h: \bigcup_{a\in\A} I(a) \to \R \]
   such that $h'(x)\geq 2$ everywhere and such that
   $h(I(a))=\R$ for each $a\in\A$ (so 
   $h|_{I(a)}:I(a)\to\R$ is a diffeomorphism). 
   Then the desired homeomorphism between
   $\R$ and the space $\Sequ$ is provided by the symbolic dynamics of the map $h$. More precisely, for
   $x\in\R$, let $\psi(x)$ be the (unique) finite or infinite sequence $\saddr=s_0 s_1\dots$ satisfying
   $h^j(x)\in I(s_j)$ (whenever defined). It is easy to check that this map is indeed an order-isomorphism
   between $\R$ and $\Sequ$. 
 \end{proof}

 \smallskip

 It now remains to define a topology on the set
   $\tilde{H} := \ol{H}\cup \Sequt$. This topology will agree with the 
   induced topology on
   $\ol{H}$, so we only need to specify a neighborhood base for every 
   $\saddr\in\Sequt$. This can be done
   as follows. Denote by $H_R$ the right half plane 
   $H_R = \{a+ib: a>R, b\in\R\}$; we will always assume
   $R$ to be chosen sufficiently large that $H_R\subset H$. 

  \begin{enumerate}
   \item Let $\saddr= T_0 T_1 T_2\dots \in\A^{\N_0}$ be an infinite 
    external address, and let $n\geq 0$. Let
      $R>0$, and let $V$ be the unbounded component of $T_n\cap H_{R}$. If $R$ was chosen sufficiently large,
      then there is a branch $\psi$ of $F^{-n}$ defined on $H_R$ such that
      $U:=\psi(V)$ satisfies $F^j(U)\subset T_j$ for $0\leq j\leq n$. For all such $n$ and $R$, the 
      set 
      \[ \tilde{U} := U\cup \{\taddr\in \Sequt: \text{the first $n$ entries of $\taddr$ are $T_0 T_1  \dots T_n$}\} \]
      is a neighborhood of $\saddr$ by definition.
   \item Let $\saddr=T_0 T_1 \dots T_{n-1} S_n$ ($n\geq 0$) be an intermediate external address. Let
          $T_n^+,T_n^-\in\A$ with $T_n^-<S_n<T_n^+$, and let $T_{n+1}^+,
          T_{n+1}^-\in \A$ be arbitrary. 
          Also let $R$ be sufficiently large; then
   $H_R\setminus (\ol{F_{T_n^-}^{-1}(T_{n+1}^-)\cup F_{T_n^+}^{-1}(T_{n+1}^+)})$ 
          has a unique unbounded connected
          component $V$ that lies 
          between $F_{T_n^-}^{-1}(T_{n+1}^-)$ and $F_{T_n^+}^{-1}(T_{n+1}^+)$. Choosing
      $R$ larger, if necessary, 
       we set 
          $U := \psi(V)$, where $\psi$ is a branch of
          $F^{-n}$  chosen such that 
          $F^j(U)\subset T_j$ for $0\leq j\leq n-1$.

         For every such $T_n^-$, $T_n^+$ and $R$, the set
      \[ \tilde{U} := U\cup \{\taddr \in\Sequt:
           T_0 \dots T_{n-1} T_n^- T_{n+1}^- < \taddr < T_0 \dots T_{n-1} T_n^+ T_{n+1}^+ \text{ (lexicographically)}
              \} \]
      is a neighborhood of $\saddr$ by definition.
  \item Let $\saddr=+\infty$ (the definition for $-\infty$ is analogous), 
    let $T\in\A$ and let $\Gamma$ be a Jordan arc that connects
      $\partial T$ to $\partial H$.
      Let $U$ be the connected component of
      $\ol{H}\setminus (\ol{T}\cup \Gamma)$
      that contains points ``above'' $T$ (i.e., that contains a 
      sequence $a+ib_n$, with $a$ fixed and $b_n\to +\infty$). 

     Then, for all $T$ and $\Gamma$ as above, the set
     \[ \tilde{U} := U \cup \{+\infty\}\cup \{\taddr\in\Sequt: \text{the first entry of $\taddr$ is larger than $T$}\} \]
     is a neighborhood of $+\infty$ by definition.     
  \end{enumerate}
 It is not difficult to check that each of the neighborhoods described above is in fact an open subset of
   $\tilde{H}$. 

 \begin{prop}[Topology of $\tilde{H}$]
   Equipped with the topology described above, the space $\tilde{H}$ is homeomorphic to the closed unit disk.
 \end{prop} 
 \begin{proof}[Sketch of proof]
   By Urysohn's metrizability theorem \cite[page 125, Theorem 16]{kelley}, 
   we see that $\tilde{H}$ is metrizable. 
   Furthermore, the space is compact.
   (We only need to show that every sequence $z_n\in\ol{H}$ with $z_n\to\infty$
    in $\C$ has a convergent subsequence
   in $\tilde{H}$; this is not difficult to verify.) Since the neighborhoods given are connected, the space is
   locally connected. 

  Furthermore, let us denote by $\partial\tilde{H}$ the set $\partial H\cup \Sequt$, which is topologically a 
   circle. Suppose $\gamma$ is a ``crosscut'', that is,
   an arc that connects two points of $\partial\tilde{H}$, with all points of $\gamma$
   except for the endpoints belonging to $H$. Then one can verify that 
   $\tilde{H}\setminus\gamma$ has exactly two components, each of which intersects the boundary circle
   $\partial\tilde{H}$ in an
   arc. From this, it follows that $\tilde{H}$ is homeomorphic to the closed unit disk. (E.g.\, glue
   a disk inside the boundary circle, and use the Kline sphere 
   characterization \cite{bingkline}
   to show that the resulting space is homeomorphic to the $2$-sphere.)
 \end{proof}

 Let $T$ be a tract of $F$, and let $\tilde{T}$ be the closure of $T$ in
  $\tilde{H}$. Then it follows from the definitions that the homeomorphism
  $F:\ol{T}\to\ol{H}$ extends to a homeomorphism
  $F:\tilde{T}\to\tilde{H}$. 

\subsection*{Properties}
 The topology we introduced---and 
  particularly its analog in the original dynamical plane---has
  a number of further useful properties, again similarly to the
  setting of exponential maps. For example, any
  dynamically natural curve (e.g.\ one which does not contain escaping 
  points) that tends to infinity will have a
  unique endpoint on the boundary circle. However, we shall restrict 
  to those facts that are relevant
  to our paper.

 \begin{prop}
  Let $F\in\BlogP$, and let $\saddr$ be an external address. If $z_n\in J_{\saddr}(F)$ is
   a sequence with 
   $\re z_n\to\infty$, then $z_n\to \saddr$ in the topology of $\tilde{H}$.
 
  In particular, $J_{\saddr}(F)\cup\{\saddr\}$ is a compact subset of $\tilde{H}$. If
   $F$ is of disjoint type, then $J_{\saddr}(F)\cup\{\saddr\}$ is connected. 
 \end{prop}
 \begin{proof}
  It follows from the definition of the topology on $\tilde{H}$ that points in
   $J_{\saddr}(F)$ cannot accumulate on any element of $\Sequt$ apart from $\saddr$, which implies
   the first two claims. The fact that $J_{\saddr}(F)\cup\{\saddr\}$ is connected when $F$ is of disjoint type
   follows from the fact that $J_{\saddr}(F)$ is the countable intersection of closed, unbounded, connected
   sets by definition, and hence every connected component of $J_{\saddr}(F)$ is unbounded.
   (In fact it follows from \cite{eremenkoproperty} that $J_{\saddr}(F)$ is connected as a subset of
    $\C$, but we shall not require this.)
 \end{proof}

 \begin{prop}[Hairy subsets of $J(F)$] \label{prop:precomb}
  Let $F\in\BlogP$, and let $X\subset J(F)$ be a closed set such that 
  \begin{enumerate} 
     \item for every $\saddr\in \A^{\N_0}$, the set
       $X_{\saddr} := J_{\saddr}(F)\cap X$ is a closed arc to infinity;
   \item the set of addresses $\saddr \in \A^{\N_0}$ with
     $X_{\saddr}\neq\emptyset$ is dense in $\A^{\N_0}$. 
  \end{enumerate}
  (If $F$ satisfies a uniform head-start condition, then
    these properties are satisfied by the set $X$ from Corollary \ref{cor:absorbing}.) 

  Let us denote by $\tilde{X}$ the closure of $X$ in the space $\tilde{H}$. 
   Then $\tilde{X}=X\cup\Sequt$ is a continuum and, with $B:=\Sequt$, 
   we have:
   \begin{enumerate}
     \item \label{item:arcclosure}       
        The closure  of every component of $\tilde{X}\setminus B = X$ is an arc,
         with exactly one endpoint in $B$;
     \item \label{item:distinctcomponents}
       distinct components of $X$ have disjoint closures in $\tilde{X}$;
     \item \label{item:densehairs} the set $X$ is dense in $\tilde{X}$. 
   \end{enumerate}
 \end{prop}
 \begin{proof}
  $\tilde{H}$ is a compact metric space, and hence so is
  $\tilde{X}$. By the second assumption on $X$, we have
  $\Sequt\subset \tilde{X}$. Hence
  $\tilde{X}=X\cup \Sequt$ is connected, and 
  can be written as the
  disjoint union of the arc $B$ and the (half-open)
  rays $X_{\saddr} := J_{\saddr}(F)\cap X$, with each ray
  ending at the point $\saddr\in B$ by the previous Proposition. This establishes
  the first two claims, and the third is satisfied by definition.

 Finally, suppose that
  $F$ satisfies a uniform head-start condition and let $X$ be the set from
   Corollary \ref{cor:absorbing} and let $R>0$. By Lemma \ref{lem:periodic} (applied
   to the restriction of $F$ to $F^{-1}(H_R)$), for every periodic address $\saddr$
   there is a point $z\in J(F)$ with $\extaddr(z)=\saddr$ and
   $\re F^j(z)\geq R$ for all $j$. This shows that, in this setting, the set 
   $X$ does indeed satisfy our hypotheses.
 \end{proof}

 \subsection*{The case of finitely many tracts}

 After the preceding, somewhat abstract, discussion, let us return
  to the case where $F$ is the logarithmic transform of an entire function
  $f\in\B$ which has only finitely many, say $N$, tracts over infinity. 
  As discussed above, 
  $\A$ can be identified with the set of pairs
   $(s,j)$, where $j\in\{1,\dots,N\}$ and $s\in\Z$, where
   $(s_1,j_1)>(s_2,j_2)$ if $s_1>s_2$ or $s_1=s_2$ and $j_1>j_2$. 

 Now, extending $\A$ to the set $\ol{\A}$ involves simply adding an 
 ``intermediate entry'' between any two adjacent tracts, so we can identify $\ol{\A}$ with
  the set 
   \[ \left\{(s,j): j\in\left\{\frac{1}{2},1,\frac{3}{2},\dots,N\right\}, s\in\Z\right\}. \]
  
 \subsection*{The original dynamical plane}

 Suppose that the map $F\in\BlogP$ is a logarithmic transform
  of a function
  $f\in\B$. Then the translation $z\mapsto z + 2\pi i$ extends to
  a homeomorphism of $\tilde{H}$ to itself. We can quotient
  $\tilde{H}\setminus\{-\infty,+\infty\}$ to obtain
  a compactification of the original dynamical plane of $f$, with
  a ``circle of addresses'' at infinity. Again, this space will be
  homeomorphic to a closed unit disk, with the open disk corresponding
  to the original dynamical plane. 

\section{Proof of Theorem \ref{thm:bouquet}}\label{sec:proof}

 With Propositions \ref{prop:density} and \ref{prop:precomb}, we
  have established most of the conditions required to show that
  $J(F)$ is a {\cantorbouquet}, and hence establish Theorem
  \ref{thm:bouquet}.
  It only remains to show that, when $z_0\in J(F)$ and
  $z_n\in J(F)$ with $z_n\to z_0$, then the arcs connecting
  $z_n$ to infinity converge to the corresponding arc for $z_0$ in the
  Hausdorff metric (provided $F$
  satisfies a uniform head-start condition). 

 The main fact needed is as follows, where we use the
  speed ordering $\prec$
  as defined in Proposition \ref{prop:speedordering}.

 \begin{prop} \label{prop:nowiggle}
   Let $F\in\BlogP$ be a function that satisfies a uniform head-start 
    condition. Suppose that $a_n\in J(F)$ converges to 
    a point $a\in J(F)$, and that, for each $n$, 
    $b_n\in J_{\addr(a_n)}(F)$
    has the same external address as $a_n$ and satisfies 
    $a_n\preceq b_n$ in the speed ordering of $F$.

  If $b\in J(F)$ is an accumulation point of the sequence $b_n$, then 
     then $a\preceq b$. 
 \end{prop}
 \begin{proof}
  We prove the contrapositive: If $a\succ b$, then $a_n\succ b_n$ for
   sufficiently large $n$.

  Indeed, by definition of the speed ordering, there exists
   $k\geq 0$ such that $\re F^k(a)> \phi(\re F^k(b))$, where $\phi$ is
   the function from the head-start condition. Hence, for sufficiently
   large $n$, we also have $\re F^k(a_n)>\phi(\re F^k(b_n))$, by
   continuity, and thus $a_n\succ b_n$, as claimed. 
 \end{proof}

 \begin{prop} \label{prop:comb}
  Let $F\in\BlogP$ satisfy a uniform head-start condition, let 
   $X$ be as in Corollary \ref{cor:absorbing} and denote
   the closure of $X$ in $\tilde{H}$ by $\tilde{X}$.

  Then $\tilde{X}$ is a comb.
 \end{prop}
 \begin{proof}
   For any $z\in \tilde{X}$, let us denote by $A_z$ the unique arc connecting
   $z$ to $\Sequt$ in $\tilde{X}$. Then 
     \[ A_z = \{z\}\cup \{w\in J_{\saddr}(F): w\succ z\} \cup \{\saddr\}, \]
   where $\saddr=\extaddr(z)$. 
    By Proposition \ref{prop:precomb}, it only remains to check the requirement
   that, if a sequence $z_n\in\tilde{X}$ converges to a point $z_0\in\tilde{X}$, 
   then the arcs $A_{z_n}$ converge to $A_{z_0}$ in the Hausdorff metric. 

  Passing to a subsequence, we may assume that $A_{z_n}$ is a 
  convergent sequence in the Hausdorff metric; let the limit be
  $A'$, say. Clearly $A'$ must be a subset of the arc
  $X_{\saddr}\cup\{\saddr\}$, where $\saddr=\extaddr(z_0)$. Since $A'$ is 
  connected (as the Hausdorff limit of compact connected subsets of the
  compact space $\tilde{X}$) and contains both $z_0$ and $\saddr$, we must have
  $A_{z_0}\subset A'$. 
 
 On the other hand, Proposition \ref{prop:nowiggle} shows that
  $A'\subset A_{z_0}$, and hence $A_{z_0}=A'$. This proves that 
  $\tilde{X}$ is a comb.
 \end{proof}

 \begin{cor}[Julia set is a hairy arc] \label{cor:bouquetheadstart}
  Let $F\in\BlogP$ be of disjoint type and satisfy a uniform head-start 
   condition. Let $\tilde{J}(F)$ denote the closure of $J(F)$ in the
   topology of $\tilde{H}$. 

  Then the set $\tilde{J}(F)$ is a one-sided hairy arc. In particular, both
   $J(F)$ and $\exp(J(F))$ are {\cantorbouquet}s. 
 \end{cor}
 \begin{proof}
  By Proposition \ref{prop:comb}, $\tilde{J}(F)$ is a comb,  and 
   Proposition \ref{prop:density} thus implies that
    $\tilde{J}(F)$ is a hairy arc, which is one-sided by construction. 
    Using Theorem \ref{thm:ambient}, this easily implies  
    that $J(F)$ is a {\cantorbouquet}. 

 It is straightforward
  to see that $\exp(J(F))$ is also a {\cantorbouquet}. For example,
  we can use the above-mentioned extension of the exponential map to
  $\tilde{H}\setminus\{-\infty,+\infty\}$ to see that $J(f)$ is 
  a brush of a ``hairy circle'' in the sense of \cite{AO} (which is
  the same has a hairy arc with the two endpoints of the base identified).  
 \end{proof}

 This completes the proof of Theorem \ref{thm:bouquet}. 

\section{The general case} \label{sec:proofgeneral}

 To prove Theorem \ref{thm:bouquetgeneral}, 
  it remains to consider the case of a function $F\in\BlogP$
  that is not necessarily
  of disjoint type. We have already seen that the set $X\subset J(F)$ 
  from Corollary~\ref{cor:absorbing} can be compactified to a comb, but in general there
  is no reason to expect that every hair in $X$ has other hairs accumulating
  both from above and below. However, we shall show that $X$ at least
  contains a {\cantorbouquet}. 
  Let us define, for $R>0$:
   \[ J^R(F) := \{z\in J(F): \re F^j(z)\geq R\text{ for all $j\geq 0$}\}. \]
   Our goal now is to show:
  \begin{prop}[Hairy arcs in Julia sets] \label{prop:hairyarcgeneral}
   Let $F\in\BlogP$ satisfy a uniform head-start condition, 
    and let $\tilde{X}$ be the comb from
    Proposition \ref{prop:comb}. 
    Then 
    there exists a hairy arc $Z\subset \tilde{X}$ 
    (again with base $B=\Sequt$) such
    that
     $J^R(F)\subset Z$ for sufficiently large $R$, and such that
    $Z\setminus \Sequt$ is $2\pi i$-periodic. 
  \end{prop}

 In order to work with non-disjoint-type functions, we need the following
  well-known expansion estimate.
 \begin{lem}[{\cite[Lemma 1]{EL}}] \label{lem:eremenkolyubich}
  Let $F:\T\to H$ be an element of $\BlogP$, and let 
   $K$ be sufficiently large that
   $H_K\subset H$. Then
    \[ |F'(z)|\geq \frac{1}{4\pi}(\re F(z) - \ln K)\quad
      \text{for all $z\in\T$.} \]
  In particular, there is $R>0$ such that
   $|F'(z)|\geq 2$ whenever $\re F(z)\geq R$.
 \end{lem}

   The following is the main dynamical fact that will allow us to construct a 
   {\cantorbouquet} in $J(F)$. 
 \begin{prop}[Accumulation from above and below]\label{prop:density2}
  Let $F:\T\to H$ belong to the class $\BlogP$, and let $\tau>0$. 
  
  Then  there exists $\tau'\geq \tau$ with the following property.
   For every $z_0\in J^{\tau'}(F)$, there exist
   sequences 
   $z_n^-,z_n^+\in J^{\tau}(F)$ with
   $\addr(z_n^-)<\addr(z_0)<\addr(z_n^+)$ for all $n$ and
   $z_n^-\to z_0$, $z_n^+\to z_0$.
 \end{prop}
 \begin{proof}
  Let $R$ be the number from
   Lemma \ref{lem:eremenkolyubich}, so that
   $|F'(z)|\geq 2$ whenever $\re F(z)\geq R$; also assume that
   $R$ is large enough such that $H_R\subset H$.
   If we set
   $\tau':=\max(R,\tau)+\pi$, then the claim follows analogously to
   Proposition \ref{prop:density}, using expansion in the Euclidean
   metric instead of expansion in the hyperbolic metric. 
 \end{proof}

 We are now ready to prove Proposition \ref{prop:hairyarcgeneral}. For the remainder
  of the section, we 
  fix the function $F\in\BlogP$ (satisfying a uniform head-start condition)
  and the sets $X$ and $\tilde{X}$. To begin,
  we essentially show that the comb $\tilde{X}$ can be ``straightened'',
  using the same proof as \cite[Theorem 3.11]{AO}.

 \begin{prop}[Potential function]
  There exists a continuous function $\rho:X\to [0,\infty)$ with
  the following properties.
  \begin{enumerate}
    \item $\rho$ is $2\pi i$-periodic;
    \item $\rho$ is strictly increasing along hairs; i.e.\ if
      $z,w\in X$ have the same external address and
      $z\succ w$, then $\rho(z)>\rho(w)$;
    \item For any sequence $(z_n)$, we have
      $\rho(z_n)\to\infty$ if and only if
      $\re z_n\to\infty$. \label{item:whitneyinfinity}
  \end{enumerate}
 \end{prop}
 \begin{proof}
  For any point $x\in X$, let us denote by 
   $A_x$ the unique arc in $\tilde{X}$ that connects $x$ to the base
   $B=\Sequt$. 

  Recall that we can quotient the space $\tilde{X}\setminus \{-\infty,+\infty\}$
   by the action of $z\mapsto z+2\pi i$ (extended to $\Sequt$ in the natural
   way) to obtain a compact, connected metric space $Y$. We denote the
   projection to the quotient by 
   $\pi:\tilde{X}\setminus \{-\infty,+\infty\}\to Y$.
 
  Let $\omega:Y\to [0,1]$ be a \emph{Whitney function}---also
   referred to as a \emph{size function}---for $Y$. That is,
   $\omega$ is a continuous function on the set of
   non-empty compact subsets of $Y$ (with the Hausdorff metric), 
   taking values in $[0,1]$, and having the following two properties.
   On the one hand, 
   $\omega(\{x\})=0$ for all $x\in Y$, and on the
   other, $\omega$ is strictly increasing; i.e.\
   if $A\subsetneq B$, then $\omega(A)<\omega(B)$.
   Such a function exists for any compact metric space
   (see e.g.\ \cite[Exercise 4.33]{nadler}).

  We define 
    \[ \rho:X\to [0,\infty), \quad \rho(x) := \frac{1}{\omega(\pi(A_x))} \]
   and claim that $\rho$ has the desired properties. Indeed,
   the function is $2\pi i$-periodic by definition, and the
   fact that $\rho$ is increasing along hairs follows from the fact that
   $\omega$ is an increasing function. Finally, the ``only if'' direction in
   (\ref{item:whitneyinfinity}) follows from the first property
   of a Whitney function, while the ``if'' follows from the fact that
   $\tilde{X}$ is a comb. 
 \end{proof}

 \begin{proof}[Proof of Proposition \ref{prop:hairyarcgeneral}]
  Let $\rho$ be the function from the preceding Proposition, let
   $\tau$ be sufficiently large that $J^{\tau}(F)\subset X$ and choose
   $\tau'$ according to Proposition \ref{prop:density2}. Let $K$ be such that 
   $\re z\geq \tau'$ whenever $\rho(z)\geq K$, and $R>0$ such that
   $\rho(z)\geq K$ whenever $\re z \geq R$. 

  We define 
    \[ Z := \{z\in X: \rho(F^j(z))\geq K\text{ for all $j\geq 0$}\}\cup
           \Sequt. \] 
  Since $\rho$ is a continuous function, it follows that
   $Z$ is a compact subset
   of $\tilde{X}$. The fact that $\rho$ is increasing along hairs
   implies that $Z$ is a comb. The set $J^R(F)$ is contained in $Z$ by 
   definition.
   It remains to show that 
   $Z$ is a hairy arc. 

  So let $z_0\in Z$. Let us assume first that $z_0$ is not
   an endpoint of $Z$, so that there is a point $w_0\in Z$ with
   $\addr(w_0)=\addr(z_0)$ and $w_0\prec z_0$. So
   there is some $M\geq 0$ such that $\re F^{M}(z_0)> \phi(\re F^{M}(w_0))$,
   where $\phi$ is the function from the head-start condition. 
   
  Since $w_0\in Z$, we have $w_0\in J^{\tau'}(F)$. Thus we can apply
   Proposition \ref{prop:density2}: there exist
   sequences $w_n^-,w_n^+\in J^{\tau}(F)\subset X$, converging to $w_0$, with 
   $\addr(w_n^-)<\addr(w_0)<\addr(w_n^+)$. Since $X$ is a comb, we can pick
   $z_n^{\pm}\succ w_n^{\pm}$ such
   that $z_n^-\to z_0$ and $z_n^+\to z_0$. By continuity, we have
   $\re F^M(z_n^{\pm})>\phi(\re F^j(w_n^{\pm}))$ for sufficiently
   large $n$; without loss of generality, we may assume this holds for
   all $n$. Using the head-start condition, we have
   \[ \re F^j(z_n^{\pm})>\phi(\re F^j(w_n^{\pm}))\geq \re F^j(w_n^{\pm}) \]
   for all $j\geq M$. Furthermore, let $\eps>0$ be such that
   $|z_n^{\pm}-w_n^{\pm}|\geq \eps$ for all $n$. Since $F$ is expanding,
   we have 
    \[ |F^j(z_n^{\pm})-F^j(w_n^{\pm})|\geq 2^j\cdot \eps. \]

  Recall that the union $\T$
   of tracts of $F$ is $2\pi i$-periodic, that each tract is
   disjoint from its $2\pi i$-translates and that they
   accumulate only at $\infty$. Hence if  $C>0$ is sufficiently large,
   then the following holds: If 
   $w\in \ol{\T}$ and $z\in\ol{\T}$ with $\re z \geq \re w$ and
   $|z-w|\geq C$, then $\re w\geq R$. 

  So if $M_1\geq M$ with
   $2^{M_1}\geq C/\eps$, then  
   \[ \re(F^j(z_n^{\pm})\geq R, \quad\text{and hence}\quad
      \rho(F^j(z_n^{\pm}))\geq K \]
   for all $n$ and all $j\geq M_1$. Also, if $n$ is large enough, then
   $\rho(F^j(z_n^{\pm}))>K$ for $j=0,\dots,M_1-1$ (by continuity).
   Thus $z_n^{\pm}\in Z$ for sufficiently large $n$, as desired.
 
  If $z_0$ is an endpoint of $Z$, then we can pick a sequence of points
   on the hair of $z_0$. Each of these is not an endpoint and 
   hence we can apply the fact 
   we have just proved. By diagonalization, we can thus find sequences
   $z^-, z^+\in Z$ with $z^{\pm}\to z_0$ and 
   $\extaddr(z^-)<\extaddr(z_0)<\extaddr(z^+)$. 

  Thus we have verified that the comb $Z$ is indeed a hairy arc. 
 \end{proof}

 As in the previous section, we obtain the following Corollary. 
 \begin{cor}
  Let $F\in\BlogP$. Then the set $X$ from Corollary \ref{cor:absorbing} can
   be chosen such that $X$ and $\exp(X)$ are {\cantorbouquet}s. 
 \end{cor}

 This completes the proof of Theorem \ref{thm:bouquetgeneral}.

 \subsection*{Acknowledgments}
  We would like to thank Helena Mihaljevi\'c-Brandt for interesting 
   discussions.


\begin{thebibliography}{99}
\bibitem{AO} J.~M.~Aarts and L.~G.~Oversteegen, \emph{The geometry of
Julia sets}, Trans. Amer. Math. Soc. {\bf 338} (1993), 897--918.

\bibitem{Ba} I.~N.~Baker, \emph{The domains of normality of an entire
function}, Ann. Acad. Sci. Fenn. Ser. A I Math. {\bf 1} (1975), 277--283. 

\bibitem{B} K.~Bara\'nski, \emph{Trees and hairs for some hyperbolic entire maps of finite order}, Math Z. {\bf 257} (2007), 33--59. 

\bibitem{B2} \bysame, 
 \emph{Hausdorff dimension of hairs and ends for entire maps of finite order}, Math. Proc. Cambridge Philos. Soc. {\bf 145} (2008), 719--737.

\bibitem{bingkline} R.~H.~Bing, 
\emph{The Kline sphere characterization problem}, 
Bull. Amer. Math. Soc. \textbf{52} (1946), 644--653. 

\bibitem{BD} C.~Bodel\'on, R.~L.~Devaney, L.~Goldberg, M.~Hayes,
J.~Hubbard and G.~Roberts, \emph{Hairs for the complex exponential
family}, Internat. J. Bifur. Chaos Appl. Sci. Engrg. {\bf 9} (1999),
1517--1534.

\bibitem{BDDJM} R.~Bhattacharjee, R.~L.~Devaney, R.~E.~L.~DeVille, K.~Josi{\'c} and M.~Moreno-Rocha, \emph{Accessible points in the Julia sets of stable
exponentials}, Discrete Contin. Dyn. Syst. Ser. B {\bf 1}
(2001), 299--318.

\bibitem{BO} W.~T.~Bula and L.~G.~Oversteegen, 
  \emph{A characterization of smooth Cantor bouquets}, Proc.\ Amer.\ Math.\
  Soc.\ \textbf{108} (1990),529--534. 
  
\bibitem{CG} L.~Carleson and T.W. Gamelin, {\emph Complex Dynamics}, Springer-Verlag New York Inc., 1993.
  

\bibitem{DK} R.~L.~Devaney and M.~Krych, \emph{Dynamics of $\exp(z)$},
Ergodic Theory Dynam. Systems {\bf 4} (1984), 35--52.

\bibitem{DT} R.~L.~Devaney and F.~Tangerman, \emph{Dynamics of entire functions near the essential singularity}, Ergodic Theory Dynam. Systems {\bf 6} (1986), 489--503.

\bibitem{EL} A.~E.~Eremenko and M.~Y.~Lyubich, \emph{Dynamical properties of some classes of entire functions}, Ann. Inst. Fourier (Grenoble) {\bf 42}
(1992), 989--1020.

\bibitem{H} W.~K.~Hayman, \emph{Subharmonic Functions}, Vol. 2, London Mathematical Society Monographs {\bf 20}, Academic Press Inc., London, 1989. 

\bibitem{K} B.~Karpi\'nska, \emph{Hausdorff dimension of the hairs
without endpoints for $\lambda\exp z$}, C. R. Acad. Sci. Paris
Ser. I Math. {\bf 328} (1999), 1039--1044. 

\bibitem{kelley}
J.~L.~Kelley, \emph{General topology},
    D. Van Nostrand Company, Inc., Toronto-New York-London, 1955.


\bibitem{mcmullen}
 C.\ McMullen, \emph{%
Area and Hausdorff dimension of Julia sets of entire functions}, 
Trans. Amer. Math. Soc. \textbf{300} (1987), no. 1, 329--342.

\bibitem{helenalanding}
 H.~Mihaljevi\'c-Brandt, \emph{A landing theorem for dynamic rays
  of geometrically finite entire functions}, J. London Math. Soc. (2)
  \textbf{81} (2010), 696--714.

\bibitem{helenaconjugacy}
\bysame, \emph{Semiconjugacies, pinched Cantor bouquets
   and hyperbolic orbifolds}, Trans.\ Amer.\ Math.\ Soc. (to appear), 
   arXiv:0907.5398. 


\bibitem{nadler}
 S.~Nadler, \emph{Continuum theory. An introduction},
  Monographs and Textbooks in Pure and Applied Mathematics, 158,
   Marcel Dekker, Inc., New York, 1992. 

\bibitem{R1} L.~Rempe, \emph{Topological dynamics of exponential maps on their escaping sets}, Ergodic Theory Dynam. Systems {\bf 26} (2006), 1939--1975. 

\bibitem{nonlanding} \bysame, 
 \emph{On nonlanding dynamic rays of exponential maps},
     Ann. Acad. Sci. Fenn. Math.  32  (2007),  no. 2, 353--369.

\bibitem{eremenkoproperty} \bysame,
 \emph{On a question of Eremenko concerning escaping components of entire functions},  
   Bull. London Math. Soc. \textbf{39} (2007), 661--666.

\bibitem{R} \bysame, \emph{Rigidity of escaping dynamics
for transcendental entire functions}, Acta Math. {\bf 203} (2009), 235--267.

\bibitem{hypdim_two} \bysame, \emph{Hyperbolic entire functions with full hyperbolic dimension},
  Manuscript, 2010. 

\bibitem{expcombinatorics} L.~Rempe and D.~Schleicher,
 \emph{Combinatorics of bifurcations in exponential parameter space},
P.~J.~Rippon, (ed.) et al., Transcendental dynamics and complex analysis. In honour of Noel Baker, London Math. Soc. Lecture Note Ser. {\bf 348}, 317--370, Cambridge Univ. Press, Cambridge, 2008. 

\bibitem{RRRS} L.~Rempe, G.~Rottenfusser, J.~R\"uckert and D.~Schleicher, \emph{Dynamic rays of bounded-type entire functions}, 
Ann. of Math. \textbf{173} (2011), no. 1, 77-125.

\bibitem{rempestallard}
  L.~Rempe and G.~.M.~Stallard, 
  \emph{Hausdorff dimension of escaping sets of transcendental entire functions},
    Proc.\ Amer.\ Math.\ Soc.\ \textbf{138} (2010), no. 5, 1657--1665. 

\bibitem{RS} G.~Rottenfusser and D.~ Schleicher,
\emph{Escaping points of the cosine family}, P.~J.~Rippon, (ed.) et al., Transcendental dynamics and complex analysis. In honour of Noel Baker, London Math. Soc. Lecture Note Ser. {\bf 348}, 396--424, Cambridge Univ. Press, Cambridge, 2008. 
\bibitem{SZ} D.~Schleicher and J.~Zimmer, \emph{Escaping points of exponential maps}, J. London Math. Soc. (2) {\bf 67} (2003), 380--400.

\bibitem{SZ1} D.~Schleicher and J.~Zimmer, \emph{Periodic points and dynamic rays of exponential maps}, Ann. Acad. Sci. Fenn. Math. {\bf 28} (2003), 327--354.

\bibitem{stallard} G.~M.~Stallard, 
\emph{The Hausdorff dimension of Julia sets of entire functions. IV.}
J. London Math. Soc. (2) \textbf{61} (2000), no. 2, 471--488. 

\bibitem{UZ} M.~Urba\'nski and A.~Zdunik,
 \emph{The finer geometry and dynamics of the hyperbolic exponential family},
    Michigan Math. J. \textbf{51} (2003), no. 2, 227--250.

\end{thebibliography}
\end{document}